\documentclass[reqno,10.5pt]{amsart}
\usepackage{mathtools}

\mathtoolsset{showonlyrefs=true}

\usepackage[a4paper,left=35mm,right=35mm,top=30mm,bottom=30mm,marginpar=25mm]{geometry}

\usepackage{amsmath}
\usepackage{amssymb}
\usepackage{amsthm}
\usepackage{eurosym}
\usepackage[dvips]{graphics}
\usepackage{graphicx}
\usepackage{epsfig}
\usepackage[colorlinks]{hyperref}

\usepackage{dsfont}
\usepackage[nocompress, space]{cite}

\usepackage[displaymath,mathlines]{lineno}
\usepackage{comment}

\allowdisplaybreaks

\usepackage{hyperref}

\usepackage{ifthen}

\newcommand{\essinf}{\operatorname{essinf}}
\renewcommand{\div}{\operatorname{div}}

\newcommand{\Rr}{{\mathbb{R}}}

\newcommand{\Nn}{{\mathbb{N}}}

\newcommand{\Tt}{{\mathbb{T}}}

\newcommand{\Hh}{{\overline{H}}}

\newcommand{\Dd}{{\mathcal{D}}}

\newcommand{\RR}{{\mathbb R}}
\newcommand{\EE}{{\mathbb E}}

\newcommand{\td}{{\mathbb{T}^d}}
\newcommand{\eps}{\varepsilon}

\newcommand{\epsi}{\varepsilon}

\def\leq{\leqslant}
\def\geq{\geqslant}

\numberwithin{equation}{section}

\newtheoremstyle{thmlemcorr}{10pt}{10pt}{\itshape}{}{\bfseries}{.}{10pt}{{\thmname{#1}\thmnumber{
                        #2}\thmnote{ (#3)}}}
\newtheoremstyle{thmlemcorr*}{10pt}{10pt}{\itshape}{}{\bfseries}{.}\newline{{\thmname{#1}\thmnumber{
                        #2}\thmnote{ (#3)}}}
\newtheoremstyle{defi}{10pt}{10pt}{\itshape}{}{\bfseries}{.}{10pt}{{\thmname{#1}\thmnumber{
                        #2}\thmnote{ (#3)}}}
\newtheoremstyle{remexample}{10pt}{10pt}{}{}{\bfseries}{.}{10pt}{{\thmname{#1}\thmnumber{
                        #2}\thmnote{ (#3)}}}
\newtheoremstyle{ass}{10pt}{10pt}{}{}{\bfseries}{.}{10pt}{{\thmname{#1}\thmnumber{
                        A#2}\thmnote{ (#3)}}}

\theoremstyle{thmlemcorr}
\newtheorem{theorem}{Theorem}
\numberwithin{theorem}{section}

\newtheorem{corollary}[theorem]{Corollary}

\newtheorem{assumption}[theorem]{Assumption}

\theoremstyle{thmlemcorr*}
\newtheorem{theorem*}{Theorem}
\newtheorem{lemma*}[theorem]{Lemma}
\newtheorem{corollary*}[theorem]{Corollary}
\newtheorem{proposition*}[theorem]{Proposition}
\newtheorem{problem*}[theorem]{Problem}
\newtheorem{conjecture*}[theorem]{Conjecture}

\theoremstyle{defi}
\newtheorem{definition}[theorem]{Definition}

\newtheorem{problem}{Problem}

\theoremstyle{remexample}
\newtheorem{remark}[theorem]{Remark}

\newtheorem{lem}[theorem]{Lemma}

\theoremstyle{ass}

\begin{document}
        
        \title[Weak-strong uniqueness for solutions to MFGs]{Weak-strong uniqueness\\ for solutions to mean-field games
        }
        
        \author{Rita Ferreira}
        \address[R. Ferreira]{
                King Abdullah University of Science and Technology (KAUST),
                CEMSE Division, Thuwal 23955-6900, Saudi Arabia.}
        \email{rita.ferreira@kaust.edu.sa}
        \author{Diogo Gomes}
        \address[D. Gomes]{
                King Abdullah University of Science and Technology (KAUST),
                CEMSE Division, Thuwal 23955-6900, Saudi Arabia.}
        \email{diogo.gomes@kaust.edu.sa}
        \author{Vardan Voskanyan}
        \address[V. Voskanyan]{TurinTech AI,
                Citypoint, 1 Ropemaker Street, London, EC2Y 9AW, United Kingdom.
        }
        \email{vartanvos@gmail.com}
        
        \thanks{The research reported in this publication was supported by funding from King Abdullah University of Science and Technology (KAUST).
D. Gomes was supported by King Abdullah University of Science and Technology (KAUST) baseline funds and KAUST OSR-CRG2021-4674. 
        }
        \date{\today}
        
        \begin{abstract}
                This paper addresses the crucial question of solution uniqueness in stationary first-order Mean-Field Games (MFGs). Despite well-established existence results, establishing uniqueness, particularly for weaker solutions in the sense of monotone operators, remains an open challenge. 
                Building upon the framework of monotonicity methods,
                we introduce a linearization method that enables us to prove a weak-strong uniqueness result for stationary MFG systems on the $d$-dimensional torus. In particular, we give
                explicit conditions under which this uniqueness holds. 
        \end{abstract}
        
        \maketitle
        
        \section{Introduction}
        
        Mean-field games (MFGs) offer a powerful mathematical framework for modeling and analyzing systems with large numbers of interacting agents. These models have applications in economics, engineering, and applied mathematics. This paper examines stationary first-order MFGs, corresponding to problems where deterministic dynamics drive individual agents. While the existence of solutions in many of these models is well established (e.g., through monotone operators), the question of uniqueness, especially under weaker solution frameworks, remains a challenge. 
        Our primary aim is to address this challenge.
        This paper introduces a linearization method that establishes a weak-strong uniqueness result for stationary first-order MFGs. More specifically, we consider the following first-order stationary mean-field game.  
        \begin{problem}
                \label{P1}
                Let $\mathbb{T}^{d}$ be the $d$-dimensional torus, and let \(\mathbb E\) be either \(\Rr^+\) or \(\Rr^+_0\). Fix a $C^2$ Hamiltonian
                $
                H: \mathbb{T}^{d} \times \mathbb{R}^{d} \times \mathbb{E} \rightarrow
                \mathbb{R}$.
                Find \(u\colon \Tt^d\to\RR\) and \(m\colon \Tt^d\to\EE\) satisfying the following stationary MFG system in \(\Tt^d\):
                \begin{equation}\label{eq:mainstationary}
                        \begin{cases}
                                -u-H\left(x, D u,  m \right)=0 \\
                                m-\operatorname{div}\left(m D_{p} H\left(x, Du, m\right)\right)=1.
                        \end{cases}
                \end{equation}
        \end{problem}
        
        An example Hamiltonian in Problem~\ref{P1} with \(\EE=\RR^+_0\) is
        \begin{equation}
                \label{Hpower}
                \begin{aligned}
                        H(x, p, m) = \frac{(1+|p|^2)^{\gamma/2}}{\gamma}-g(m) \quad \hbox{ for some } \gamma>1,
                \end{aligned}
        \end{equation}
        with  $g:\Rr^+_0\to \Rr $   an increasing function. In the   \(\EE=\RR^+\) case, a  typical
        Hamiltonian  arises in MFGs with congestion and is of the form
        \begin{equation}
                \label{Hconges}
                \begin{aligned}
                        H(x, p, m) = \frac{(1+|p|^2)^{\gamma/2}}{\gamma m^{\alpha}} \quad \hbox{for some
                        } \gamma>1 \text{ and } \alpha>0.
                \end{aligned}
        \end{equation}
        
        For second-order MFGs
        with an additional regularizing Laplacian term, 
        the existence of a strong (smooth) solution 
        for the discounted problem was shown, for example, in \cite{GM} for MFGs with logarithmic couplings, 
        and later extended for MFGs with congestion through an explicit transformation in 
        \cite{GMit}, and using PDE methods \cite{EFGNV2017}, and \cite{PV15}.
        For degenerate elliptic problems, this matter was addressed in 
        \cite{MR4132070} and in \cite{bakaryan2023c} for elliptic MFGs with limited regularity.
        In \cite{MR4175148}, the existence of smooth solutions for Problem \ref{P1}
        was established under suitable conditions that give an a priori lower bound on $m$.
        This problem was also addressed in \cite{GraCard} using variational methods
        and in  \cite{CardPorrLongTimeME} in the context of long-time limits. 
        For first-order MFGs with congestion,
        the existence of a solution was examined in 
        \cite{EvGom} and explicit solutions in the one-dimensional case were studied in \cite{GNPr216}. Congestion models were also studied in the context of crowd motion in 
        \cite{Achdou2019}. 
        For time-dependent problems, congestion models were examined in \cite{Achdou2016}. 
        Other approaches to congestion include density constraints \cite{San12}, \cite{San16}, and \cite{Briceno-Arias}, singular MFGs \cite{2016arXiv161107187C}, and mean-field type control \cite{AL16}. 
        The use of monotonicity methods in MFGs was introduced in \cite{almulla2017two}.
        Existence was shown in
        \cite{FG2} (also see \cite{FGT1, FeGoTa21}) using monotonicity methods, 
        which we describe next. 
        
        Following  \cite{FG2}, we consider the  operator associated with Problem~\ref{P1} given by
        \begin{equation}\label{eq:operator}
                A\left[\begin{matrix}                \eta \\
                        v
                \end{matrix}\right] \\
                :=\left[\begin{array}{c}
                        -v-H\left(x, D v,\eta\right) \\
                        \eta-\operatorname{div}\left(\eta D_{p} H\left(x, D v, \eta
                        \right)\right)-1
                \end{array}\right].
        \end{equation}
        We recall the definition of a weak solution induced by monotonicity.
        \begin{definition}\label{def:weaksol}
                For $r, \gamma\geq 1$, a weak solution to Problem \ref{P1} in $L^r\left(\td\right) \times W^{1,\gamma}\left(\td\right)$ is a pair $(m,
                u) \in L^r\left(\td\right) \times W^{1,\gamma}\left(\td\right)$
                with $m \geqslant 0$ that satisfies the variational inequality
                \begin{equation}\label{eq:weaksol}
                        \left\langle\left[\begin{matrix}
                                \eta \\
                                v
                        \end{matrix}\right]-\left[\begin{matrix}
                                m \\
                                u
                        \end{matrix}\right], A\left[\begin{matrix}
                                \eta \\
                                v
                        \end{matrix}\right]\right\rangle\geq 0
                \end{equation}
                for all   $(\eta, v) \in {C^\infty}(\Tt^d;\EE) \times
                {C^\infty}(\Tt^d)$, where \[\left\langle \left[\begin{matrix}
                        \eta \\
                        v
                \end{matrix}\right], 
                \left[
                \begin{matrix}
                        \tilde \eta \\
                        \tilde v
                \end{matrix}\right]\right\rangle=\int_{\td} (\eta \tilde \eta +v \tilde v)\, dx.\]            
        \end{definition}
        
        Using the monotonicity of the operator $A$, which holds under standard conditions on the Hamiltonian,  the authors in \cite{FG2} constructed a regularized problem with suitable bounds. Then, using 
        Minty's method, the authors proved that 
        Problem \ref{P1} has weak solutions; see Corollary 6.3 in \cite{FG2}.
        
        Alternatively, consider the following definition of a strong solution.
        \begin{definition}\label{def:ss} For 
                $r_1,\gamma_1\geq 1$,    a strong solution to Problem \ref{P1} in $L^{r_1}\left(\td\right) \times W^{1,\gamma_1}\left(\td\right)$ is a pair $(m,
                u) \in L^{r_1}\left(\td\right)  \times W^{1,\gamma_1}\left(\td\right)$,
                with $m\in\EE$  a.e.~$\in\Tt^d$,  such that \eqref{eq:mainstationary} holds in $\Dd'(\td)$.
        \end{definition}
        Strong solutions as in Definition~\ref{def:ss} were shown to exist in \cite{MR4175148}. Furthermore, the techniques in 
        \cite{ll1,ll2, ll3} show the uniqueness of strong solutions. 
        However, these proofs, based on the operator's monotonicity, do not generalize to weak solutions due to lack of integrability. 
        
        While the existence of solutions is often established, the uniqueness of solutions, particularly under a weaker solution framework, remains an open question. This paper bridges this gap by proving a
        weak-strong uniqueness result: strong solutions agree with the weak ones. This is the content of 
        the following theorem, where we refer to Section~\ref{sect:prem} for the exposition and discussion of the assumptions therein.
        
        \begin{theorem}
                \label{T1}
                Let  $( \tilde{m}, \tilde{u})  \in L^{r}(\td)\times W^{1,
                        \gamma}(\td)$  
                be a weak solution and 
                $( {m}, {u})  \in L^{r_1}(\td)\times W^{1,
                        \gamma_1}(\td)$ be a strong solution  to Problem~\ref{P1}. 
                Suppose that  Assumptions~\ref{a-4}, \ref{asmp:monotonicity}, \ref{a-2}, \ref{a-3}, \ref{a-1} hold. 
                Moreover, suppose that either
                \begin{itemize}
                        \item [(a)] Assumptions  \ref{intA}  and \ref{A3mon} hold, or
                        \item [(b)] Assumptions \ref{A3altbis} and \ref{lastone}  hold.
                \end{itemize}
                Then, we have that
                \begin{equation}\label{eq:sweq}
                        u=\tilde{u} \text{ and }\, m=\tilde{m}.
                \end{equation}
        \end{theorem}
        
        While the proof of Theorem~\ref{T1} relies on specific technical assumptions, Remark \ref{rmk:relwithapp} demonstrates its applicability to MFGs modeled by the widely-used Hamiltonians in \eqref{Hpower} and \eqref{Hconges}, provided that strong solutions exhibit sufficient regularity.
        
        Before proving that theorem, 
        we briefly explain the
        weak-strong uniqueness formal argument below. 
        This provides an overview of the approach before the technical details.
        
        Let $H$ be a Hilbert space and $F:H\to H$ a monotone operator. Here, we assume that $F$ is differentiable and denote its derivative by $F'$. Because $F$ is monotone, we have
        \begin{equation}
                \label{lmon}
                (F'(x)z, z)\geq 0
        \end{equation}
        for all $x,z\in H$. Now consider a solution $x$ of $F(x)=0$ and a weak solution $\tilde x$; that is, $(F(y),y-\tilde x)\geq 0$ for all $y\in H$. In this definition of
        weak solution, we set $y=x+\epsilon z$. Accordingly, we obtain
        \[
        (F(x+\epsilon z), x+\epsilon z-\tilde x)\geq 0.
        \]
        Because the function $\epsilon\mapsto (F(x+\epsilon z), x+\epsilon z-\tilde x)$ vanishes at $\epsilon=0$, it has a minimum there. Therefore, its derivative must vanish at $\epsilon=0$, 
        which gives
        \[
        (F'(x) z, x-\tilde x)=0. 
        \]
        From the preceding identity, we have the following potential uniqueness mechanisms: first setting $z=x-\tilde x$, we have
        \[
        (F'(x) (x-\tilde x), x-\tilde x)=0. 
        \]
        So, to get uniqueness, it is enough to show a stronger version of \eqref{lmon}; more precisely,  that $(F'(x)z, z)=0$ implies $z=0$. This approach is used to prove Theorem \ref{T1} under conditions (a). Alternatively, we can show that for any $w\in H$, there exists a unique solution $z$ of $F'(x)z=w$. In this case, we obtain $(w,x-\tilde x)=0$, which implies $x=\tilde x$ since $w$ is arbitrary. This is the idea behind the proof of Theorem \ref{T1} under conditions (b).
        A significant part of the proof involves developing these ideas under low-regularity conditions, which requires a delicate analysis.

        \section{Assumptions and preliminary results}\label{sect:prem}
        
        Next, we present the assumptions needed to establish our weak-strong uniqueness result, stated in Theorem~\ref{T1}. We first require strong solutions to be at least as regular as weak solutions. 
        \begin{assumption}
                \label{a-4}     
                The parameters $(r,\gamma)$ and $(r_1,\gamma_1)$ in the definition of weak and strong solutions satisfy $r_1\geq r>1$ and $\gamma_1\geq \gamma>1$.
        \end{assumption}
        
        Our proof requires a uniform lower bound on the density, $m$,  of strong
        solutions, as stated in the next assumption.
        \begin{assumption}\label{asmp:monotonicity}
                Consider a strong solution  $(m,u)$  to Problem~\ref{P1} in $L^{r_1}\left(\td\right)
                \times W^{1,\gamma_1}\left(\td\right)$.
                There exists a constant, \(c_0>0\), such that \(m\geq c_0\) on \(\Tt^d\).%
        \end{assumption}%
           \begin{remark}
                The preceding assumption requires uniform density bounds for strong solutions. These
                are known in some cases; see, for example, \cite{MR4175148}. On the other hand,  without positivity constraints on $m$, uniqueness may fail, see the example in 
                \cite{Gomes2016b} where $u$ is not unique in the regions where $m$ vanishes.
        \end{remark}

        We continue with a stronger integrability assumption on strong solutions that involves a convenient, smooth approximation and is intimately related to bounds on the Hamiltonian. In Remark \ref{rmk:power}, we address the feasibility of this assumption for the Hamiltonians in \eqref{Hpower} and \eqref{Hconges}. Here, and in the sequel,  $\gamma'$
and $ r'$  denote the conjugate exponents of $\gamma$
and $r$, given by          \(
                \frac{1}{\gamma'} + \frac{1}{\gamma} = 1\) and \( \frac{1}{r'}
+ \frac{1}{r}
                = 1  \), respectively.
        
        \begin{assumption}
                \label{a-2}
                Consider a strong solution  $(m,u)$  to Problem \ref{P1} in $L^{r_1}\left(\td\right)
                \times W^{1,\gamma_1}\left(\td\right)$. There exists \(\epsi_0>0\) such that for every \(0<\epsi\leq \epsi_0\) and \((\eta, v)\in L^{r_1}\left(\td;\Rr^+\right)
                \times W^{1,\gamma_1}\left(\td\right) \) with
                \begin{equation*}
                        \begin{aligned}
                                \Vert (m,u) - (\eta, v)\Vert_{L^{r_1}\left(\td\right)
                                        \times W^{1,\gamma_1}\left(\td\right)} \leq \epsi
                        \end{aligned}
                \end{equation*}
                and $\essinf \eta>0$, 
                we have \(H(\cdot, Dv, \eta ) \in L^{{r'}}(\mathbb{T}^d)\), \(\eta D_pH(\cdot,Dv,\eta)\in L^{\gamma'}(\td)\), and we can find a sequence  \((\eta_k, v_k)_{k\in\Nn}
                \subset C^\infty(\td;\Rr^+)\times C^\infty(\td)\)  
                such that 
                \begin{itemize}
                        \item[(a)]  \(\eta_k \to \eta \) in \(L^{r_1}(\td)\), \(v_k \to v \) in \(W^{1,
                                \gamma_1}(\td)\);
                        \item[(b)]  \(H(\cdot, Dv_k, \eta_k)\to H(\cdot,Dv,\eta)\)  in $L^{ r'}(\td)$.
                        \item[(c)] \(\eta_kD_pH(\cdot, Dv_k, \eta_k)\to \eta D_pH(\cdot,Dv,\eta)\)
                        in $L^{\gamma'}(\td)$.
                \end{itemize}
        \end{assumption}
        
        \begin{remark}\label{asmp:intergability1} 
                Let $\gamma>1$ and $r>1$ be as in the definition of weak solution. Then, under the preceding assumption, 
                we have $H(\cdot, Du, m) \in L^{{r'}}(\mathbb{T}^d)$, which together with
                \eqref{eq:mainstationary} yields $u\in L^{{r'}}(\mathbb{T}^d)$. Similarly, the prior assumption gives $m |D_p H(\cdot, Du, m)| \in L^{{\gamma'}}(\mathbb{T}^d)$.
                Moreover, using the above integrability and the monotonicity of the operator, we see that $(m,u)$ is a weak solution in $L^r\left(\td\right) \times W^{1,\gamma}\left(\td\right)$ in the sense of Definition \ref{def:weaksol}.
        \end{remark}
        
        Next, we recall the Vitali--Lebesgue
        convergence theorem (cf. \cite[Theorem~2.24]{FoLe07},
        for instance). This result plays a key role in the subsequent remark in which we address instances for which Assumption~\ref{a-2} holds. 
        
        \begin{theorem}[Vitali--Lebesgue]\label{thm:VL}
                Let  \((v_n)_{n\in\Nn}\subset L^1(\Tt^d)\) be a sequence converging a.e.~in \(\Tt^d\) for some measurable function \(v\). Assume that for all \(n\in\Nn\), \(|v_n|\leq |w_n|\) a.e.~in
                \(\Tt^d\), where   \((w_n)_{n\in\Nn}\subset L^1(\Tt^d)\) is  such that \(w_n\to w\) in \(L^1(\Tt^d)\) for some  \(w\in L^1(\Tt^d) \).
                Then, \(v\in    L^1(\Tt^d)\) and \(v_n\to v\) ~in \(L^1(\Tt^d)\).\end{theorem}
        
        \begin{remark}\label{rmk:power}
                Consider the Hamiltonian \(H\) in \eqref{Hpower}. Fix \(\epsi>0\) and let
                \((m,u)\) and \((\eta, v)\) be as in Assumption~\ref{a-2}. Assume  that  there exists a constant \(c>0\) such that the increasing function \(g\)  in \eqref{Hpower}
                satisfies\begin{equation}\label{eq:ggrowth}
                        \begin{aligned}
                                g(\theta) \leq c(1+\theta^r) \quad \text{for all \(\theta\geq0\)}.
                        \end{aligned}
                \end{equation}
                Assume further
                that
                \begin{equation}\label{eq:relpar}
                        \begin{aligned}
                                \gamma_1 \geq \gamma r' \quad \text{and } \quad r_1 \geq \max\{rr',r\gamma'\}.
                        \end{aligned}
                \end{equation}
                
                Using  \eqref{Hpower} and \eqref{eq:ggrowth}, we can find a positive constant, \(C\), independent of \((x,p,\theta)\), for which we have that
                \begin{equation*}
                        \begin{aligned}
                                |H(x,p,\theta)| \leq C\left( 1 + |p|^\gamma + \theta^r\right)\quad \text{and} \quad \theta|D_pH(x,p,\theta)|\leq C\left( 1 + \theta|p|^{\gamma-1} + \theta\right).
                        \end{aligned}
                \end{equation*}
                By Young's inequality, we can further assume that \(C\) is such that
                \begin{equation}\label{eq:forVL}
                        \begin{aligned}
                                &|H(x,p,\theta)|^{r'} \leq C\left( 1 + |p|^{\gamma r'} + \theta^{r r'}\right)\\
                                \quad&\left(\theta|D_pH(x,p,\theta)|\right)^{\gamma'}\leq C\left( 1 +\theta^{r\gamma'}+|p|^{r'\gamma}  + \theta^{\gamma'}\right).
                        \end{aligned}
                \end{equation}
                
                Then,  \eqref{eq:relpar} and \eqref{eq:forVL} yield \(H(\cdot, Dv, \eta ) \in L^{{r'}}(\mathbb{T}^d)\) and \(\eta D_pH(\cdot,Dv,\eta)\in
                L^{\gamma'}(\td)\). On the other hand, using standard mollification arguments and    recalling that  $c_\eta=\essinf \eta>0$ and that \(H\)
                and \(D_p H\) are continuous, we can find  
                a  sequence \((\eta_k, v_k)_{k\in\Nn}
                \subset C^\infty(\td;[c_\eta,+\infty))\times C^\infty(\td)\)   satisfying
                \begin{itemize}
                        \item[(i)]  \(\eta_k \to \eta \) in \(L^{r_1}(\td)\), \(v_k \to v \) in \(W^{1,
                                \gamma_1}(\td)\);
                        \item[(ii)] \(\eta_k(x) \to \eta(x) \),  \(v_k(x) \to v(x) \), and   \(D
                        v_k(x) \to D v(x) \) for almost every  \(x\in\Tt^d\);
                        
                        \item[(iii)] \(H(x, Dv_k(x), \eta_k(x))\to H(x,Dv(x),\eta(x))\)  for all \(x\in\Tt^d\);
                        \item[(iv)] \(\eta_k(x)D_pH(x, Dv_k(x), \eta_k(x))\to \eta (x)D_pH(x,Dv(x),\eta(x))\)
                        for all \(x\in\Tt^d\).
                \end{itemize}  
                Using these convergences and \eqref{eq:relpar}--\eqref{eq:forVL}, Theorem~\ref{thm:VL} yields 
                \begin{equation*}
                        \begin{aligned}
                                & H(\cdot, Dv_k, \eta_k)\to H(\cdot,Dv,\eta) \text{ in } L^{r'}(\td),\\
                                & \eta_k D_pH(\cdot, Dv_k, \eta_k)\to \eta D_p H(\cdot,Dv,\eta) \text{ in
                                } L^{\gamma'}(\td).
                        \end{aligned}
                \end{equation*}
                Hence, in this setting, Assumption~\ref{a-2} holds.
        \end{remark}
        
        \begin{remark}\label{rmk:cong}
                Consider the Hamiltonian \(H\) in \eqref{Hconges}, fix \(\epsi>0\), and let
                \((m,u)\) and \((\eta, v)\) be as in Assumption~\ref{a-2}. Assume further that Assumption~\ref{asmp:monotonicity} holds and
                \begin{equation*}
                        \begin{aligned}
                                \gamma_1 \geq \gamma r' \quad \text{and } \quad r_1 \geq \gamma' r.
                        \end{aligned}
                \end{equation*}
                Then, setting   $c_\eta=\essinf \eta>0$   and arguing as in the preceding remark, we can find a positive constant,
                \(C\), depending on  \(c_0^\eta=\min\{c_0,
                c_\eta\}\) but independent of \((x,p,\theta)\), such that for all \((x,p,\theta)\in\Tt^d\times \RR^d\times[c_{0}^\eta,+\infty)\), we have that
                \begin{equation*}\begin{aligned}
                                &|H(x,p,\theta)|^{r'} \leq C\left( 1 + |p|^{\gamma r'} \right)\\
                                \quad&\left(\theta|D_pH(x,p,\theta)|\right)^{\gamma'}\leq C\left( 1 +\theta^{\gamma' r}+|p|^{\gamma r'}
                                \right).
                        \end{aligned}
                \end{equation*}
                Continuing to argue as in the preceding
                remark, we conclude that Assumption~\ref{a-2} also holds in this case.
        \end{remark} 
        
        The next assumption concerns the ranges of values of the integrability exponents for the weak solution
        for which our techniques apply.
        \begin{assumption}
                \label{a-3}
                If $\gamma<d$,
                the parameters $r$ and $\gamma$ in the definition of weak solution satisfy $\gamma^*\geq r'$, where $\gamma^*$ is the Sobolev exponent of \(\gamma\), $\frac{1}{\gamma^*}=\frac{1}{\gamma}-\frac{1}{d}$.
        \end{assumption}
        
        \begin{assumption}
                \label{a-1}     
                Consider a strong solution  $(m,u)$  to Problem \ref{P1} in $L^{r_1}\left(\td\right)
                \times W^{1,\gamma_1}\left(\td\right)$. Let \(c_0>0\) and   $\epsi_0>0$ be as in Assumptions~\ref{asmp:monotonicity} and \ref{a-2},  respectively. For all \((\bar \eta, \bar v)\in C^\infty(\Tt^d)\times C^\infty(\Tt^d)\), there exists \(0<\bar \epsi\leq \epsi_0\) such that, for \(\bar m_\epsi=m+\epsi\bar\eta\) and \(\bar u_\epsi=u+\eps \bar v\) with \(\epsi\leq\bar \epsi\), it holds that
                \begin{itemize}
                        \item[(a)] $D_pH(\cdot, D\bar u_\epsi, \bar m_\epsi)\to D_pH(\cdot, D u, m) $  in $L^{r'}(\td)$ and in  \(L^{\gamma'}(\td)\), as \(\epsi\to0\);
                        
                        \item[(b)]  $D_mH(\cdot, D\bar u_\epsi, \bar m_\epsi)\to D_mH(\cdot, D u,
                        m)$ in $L^{r'}(\td)$, as \(\epsi\to0\);
                        
                        \item[(c)] $\bar m_\epsi D^2_{pp}H(\cdot, D\bar u_\epsi, \bar m_\epsi)\to mD_{pp}H(\cdot,
                        D u,
                        m)$  in  \(L^{\gamma'}(\td)\), as \(\epsi\to0\).
                        \item[(d)] 
                        $\bar m_\epsi D^2_{pm}H(\cdot,D\bar u_\epsi, \bar m_\epsi)\to mD_{pm}H(\cdot,
                        D u,
                        m)$ in  \(L^{\gamma'}(\td)\), as \(\epsi\to0\).
                \end{itemize}
        \end{assumption}
        
        \begin{remark}We note that Assumption~\ref{a-1} holds whether the Hamiltonian in Problem~\ref{P1} is of the form  \eqref{Hpower} or of the form \eqref{Hconges}. To see this, it suffices to argue as in Remarks~\ref{rmk:power} and \ref{rmk:cong}, observing that \(\inf_{\Tt^d}\bar m_\epsi>0\) for all \(\epsi>0\) sufficiently small by Assumption~\ref{asmp:monotonicity}. 
        \end{remark} 
        
        To prove Theorem~\ref{T1}-$(a)$, we need further integrability conditions encoded in the following  assumption that allow us to differentiate certain integral expressions. 
        
        \begin{assumption}
                \label{intA}    
                Let $(r,\gamma)$ and $(r_1,\gamma_1)$ be as in Assumption \ref{a-4}. Then, \[
                \frac{2}{r}+\frac 1 \gamma<1\qquad \frac{1}{r}+\frac{2}{\gamma}<1.
                \]                        
                Moreover, setting $q_1$, $q_2$, $q_3$, and $q_4$ such that
                \[
                \frac{1}{r}+\frac{1}{\gamma}+\frac{1}{q_1}=1,\qquad \frac{2}{r}+\frac{1}{q_2}=1,\qquad \frac{2}{r}+\frac{1}{\gamma}+\frac{1}{q_3}=1,\qquad \frac{1}{r}+\frac{2}{\gamma}+\frac{1}{q_4}=1,
                \]
                and considering a strong solution $(m,u)\in L^{r_1}(\td;\Rr_0^+)\times W^{1,\gamma_1}(\td)$, it holds that                       
                \begin{itemize}
                        \item[(a)] $D_pH(\cdot, Du, m)\in L^{q_1}(\td)$,\enspace   $D_mH(\cdot, Du, m)\in
                        L^{q_2}(\td)$,
                        \item[(b)] $D^2_{pp}H(\cdot, Du, m)\in
                        L^{q_4}(\td)$, \enspace$D^2_{pm}H(\cdot,Du, m)\in L^{q_3}(\mathbb{T}^d)$.
                \end{itemize}
        \end{assumption}
        
        Finally, we impose a strict monotonicity condition that is key to establishing our main result under conditions \((a)\). 
        
        \begin{assumption}
                \label{A3mon}
                Consider a strong solution  $(m,u)$ to Problem \ref{P1}
 in $L^{r_1}\left(\td\right)
                \times W^{1,\gamma_1}\left(\td\right)$.
                There exist a strictly positive function, \(\lambda:\Tt^d\to(0,\infty)\), such that
                \[
                \left[\begin{matrix}
                        mD^2_{pp}H(\cdot , Du, m) & \frac{1}{2}
                        mD^2_{pm}H(\cdot , Du, m)\\
                        \frac{1}{2} mD^2_{pm}H(\cdot, Du, m)
                        & -D_mH(\cdot, Du, m)
                \end{matrix}\right] \geq \lambda \text{ in }\td.
                \]
        \end{assumption} 
        
        Next, we introduce an alternative integrability condition to the previous two to prove Theorem~\ref{T1}-$(b)$. 
        
        \begin{assumption}  
                \label{A3altbis}
                Consider a strong solution  $(m,u)$ to Problem \ref{P1}
 in $L^{r_1}\left(\td\right)
                \times W^{1,\gamma_1}\left(\td\right)$,
                and define 
                \begin{equation}
                        \label{eq:defcoef}
                        \begin{aligned}
                                &A=A(\cdot, Du, m) = m D^2_{pp}H(\cdot, Du,    m) -  m \frac{D_{p}H(\cdot,
                                        Du,    m)\otimes D^2_{pm}H(\cdot , Du,    m)}{D_{m}H(\cdot, Du,    m)} \\
                                &\hskip57mm -\frac{D_{p}H(\cdot, Du,    m)\otimes D_{p}H(\cdot, Du,    m)}{D_{m}H(\cdot,
                                        Du,    m)},\\                        
                                &a=a(\cdot, Du,    m) = \frac{1}{D_{m}H(\cdot, Du,    m)},\\
                                &b=b(\cdot, Du,    m) =
                                -\frac{D_{p}H(\cdot, Du,    m)}{D_{m}H(\cdot, Du,    m)},\\
                                &c=c(\cdot, Du,    m) = -b(\cdot, Du,    m) + \frac{mD^2_{pm}H(\cdot x, Du,    m)
                                }{D_{m}H(\cdot, Du,    m)}.
                        \end{aligned}
                \end{equation}
                There exist  \(\sigma:\Tt^d \to [0,\infty)\) measurable, \(\tau>0\), \(q\in [\tfrac12 ,1)\), and \(\beta\in [1,+\infty] \) such that \(2\beta' \leq (2q)^*\) and 
                \begin{equation}
                        \label{e1}
                        \sigma^{\frac{q}{q-1}} \in L^{1}(\Tt^d), \quad \xi^T A\xi \geq \sigma|\xi|^2, \quad \xi^T((A^{-1})_S)^{-1} \xi\leq \tau
                        \xi^TA_S\xi,
                \end{equation}
                where \(M_S = \frac{M + M^T}{2}\) denotes the symmetric part of a matrix \(M\),    
                \begin{equation}
                        \label{e2}
                        \kappa=c^TA^{-1}c+b^TA^{-1}b + |a|\in L^\beta(\Tt^d),
                \end{equation}
                and
                \begin{equation}
                        \label{e3}
                        \left[\begin{matrix}
                                A(\cdot, Du, m) & \frac{1}{2}
                                \frac{mD^2_{pm}H(\cdot , Du, m)}{D_{m}H(\cdot, Du,    m)}\\
                                \frac{1}{2}
                                \frac{mD^2_{pm}H(\cdot , Du, m)}{D_{m}H(\cdot, Du,
                                        m)}
                                & -\frac1{D_mH(\cdot, Du, m)}
                        \end{matrix}\right] >0, 
                \end{equation}
                a.e. in $\Tt^d$. 
                
        \end{assumption}
        
        \begin{remark}\label{rmk:relwithapp}
                The preceding assumption is required to apply Theorem \ref{thm:adj} (see Appendix~\ref{sect:A}). Assumptions~~\ref{ass:onA}, \ref{ass:onA1}, and  \ref{ass:onkappa} correspond directly to \eqref{e1} and \eqref{e2}. 
                Moreover, we observe that the integral in \eqref{a7e} can be written as
                \[
                \int_{\Tt^d}
                \begin{bmatrix}
                        Du & u
                \end{bmatrix}
                \begin{bmatrix}
                        A(\cdot, Du, m) & \frac{1}{2}
                        \frac{mD^2_{pm}H(\cdot , Du, m)}{D_{m}H(\cdot, Du,    m)}\\
                        \frac{1}{2}
                        \frac{mD^2_{pm}H(\cdot , Du, m)}{D_{m}H(\cdot, Du,
                                m)}
                        & -\frac1{D_mH(\cdot, Du, m)}
                \end{bmatrix}
                \begin{bmatrix}
                        Du \\ u
                \end{bmatrix}dx, 
                \]
                which combined with \eqref{e3} gives the Assumption \ref{ass:onkernel}.
                
                We further observe that the Hamiltonians in  \eqref{Hpower} and  \eqref{Hconges}
                provide instances for which Assumption~\ref{A3altbis} holds provided that a strong solution satisfies the integrability condition encoded in \eqref{e2}, as we discuss next.
                \begin{itemize}
                        \item[(i)] If \(H\) is given by 
                        \eqref{Hconges}, then the corresponding matrix \(A\) in \eqref{eq:defcoef}
                        evaluated at a triplet \((x,p,m)\) is 
                        \begin{equation*}
                                \begin{aligned}
                                        A(x,p,m)&= m^{1-\alpha} \left(1+|p|^2\right)^{\frac{\gamma}{2} - 2} \left[ (1+|p|^2\mathbb{)I}
                                        +\left( \frac{\gamma}{\alpha} -2 \right) p\otimes p \right]\\
                                        &=m^{1-\alpha} \left(1+|p|^2\right)^{\frac{\gamma}{2} - 2}  \left[ \mathbb{I}  +|p|^2\mathbb{I}
                                        -p\otimes p +\left( \frac{\gamma}{\alpha} -1 \right) p\otimes p \right].
                                \end{aligned}
                        \end{equation*}
                        Thus, \(A\) is a symmetric matrix satisfying, for \(\gamma\geq \alpha\) and for every \(\xi\in\RR^d\),
                        \begin{equation*}
                                \begin{aligned}
                                        \xi^T A \xi \geq m^{1-\alpha}|\xi|^2,
                                \end{aligned}
                        \end{equation*}
                        where we used the fact that \(\xi^T \left(|p|^2\mathbb{I}
                        -p\otimes p  \right) \xi = |p|^2|\xi|^2 - |p\cdot\xi|^2\geq 0\) by the  Cauchy--Schwarz inequality.
                        
                        Consequently, if  \( \alpha\in(0,1]\), then any strong solution to \eqref{eq:mainstationary} satisfying Assumption~\ref{asmp:monotonicity} also satisfies \eqref{e1} (with \(\sigma\equiv c_0\) and \(\tau=1\)). Next, we analyze the feasibility of \eqref{e3}. 
                        
                        For \(\zeta=(\xi,\xi')\in \RR^d\times\RR\) and \((x,p,m)\in \Tt^d\times \RR^d \times \RR^+\), we have that
                        \begin{equation}\label{eq:e3a}
                                \begin{aligned}
                                        &\zeta^T \left[\begin{matrix}
                                                A & \frac{1}{2}
                                                \frac{mD^2_{pm}H}{D_{m}H}\\
                                                \frac{1}{2}
                                                \frac{mD^2_{pm}H}{D_{m}H}
                                                & -\frac1{D_mH}
                                        \end{matrix}\right] \zeta\\
                                        &\quad=m^{1-\alpha} \left(1+|p|^2\right)^{\frac{\gamma}{2} - 2}  \left[ |\xi|^2
                                        +|p|^2|\xi|^2
                                        - |p\cdot\xi|^2 +\left( \frac{\gamma}{\alpha} -1 \right) |p\cdot\xi|^2 \right]\\&\qquad + \gamma m (1+|p|^2)^{-1}( p\cdot \xi)\xi' + \frac{\gamma}{\alpha} m^{\alpha+1} (1+|p|^2)^{-\frac\gamma2}(\xi')^2.
                                \end{aligned}
                        \end{equation}
                        Moreover, using Cauchy's inequality, it follows that
                        \begin{equation*}
                                \begin{aligned}
                                        \gamma m (1+|p|^2)^{-1}( p\cdot \xi)\xi' &=\gamma m (1+|p|^2)^{-1}( p\cdot \xi) \left( \frac\gamma\alpha m^{\alpha+1}(1+|p|^2)^{-\frac\gamma2}  \right)^{-\frac12 + \frac12}\xi' \\
                                        &\geq - \frac{\alpha\gamma}4 m^{-\alpha +1} (1+|p|^2)^{\frac\gamma2 -2}|p\cdot\xi|^2  - \frac\gamma\alpha m^{\alpha+1} (1+|p|^2)^{-\frac\gamma2}(\xi')^2.
                                \end{aligned}
                        \end{equation*}
                        The preceding estimate, combined with \eqref{eq:e3a}, yields
                        \begin{equation*}
                                \begin{aligned}
                                        &\zeta^T \left[\begin{matrix}
                                                A & \frac{1}{2}
                                                \frac{mD^2_{pm}H}{D_{m}H}\\
                                                \frac{1}{2}
                                                \frac{mD^2_{pm}H}{D_{m}H}
                                                & -\frac1{D_mH}
                                        \end{matrix}\right] \zeta \\&\quad\geq m^{1-\alpha} \left(1+|p|^2\right)^{\frac{\gamma}{2} - 2}  \left[ |\xi|^2
                                        +|p|^2|\xi|^2
                                        - |p\cdot\xi|^2 +\left( \frac{\gamma}{\alpha} -1- \frac{\alpha\gamma}4  \right) |p\cdot\xi|^2 \right],
                                \end{aligned}
                        \end{equation*}
                        from which we conclude that  \eqref{e3} holds for a strong solution \((u,m)\) with \(m>0\) provided that \(0<\alpha \leq\frac{2}{\gamma} \left(-1+ \sqrt{1 +\gamma^2}\right) \).
                        
                        We then conclude that any strong solution to \eqref{eq:mainstationary}
                        satisfying Assumption~\ref{asmp:monotonicity} also satisfies \eqref{e1} and \eqref{e3} provided that the parameters \(\alpha>0\) and \(\gamma\geq 1\) in 
                        \eqref{Hconges} satisfy \(\alpha \leq\min\left\{1,\frac{2}{\gamma} \left(-1+ \sqrt{1
                                +\gamma^2}\right)\right\} \).
                        As observed before,  condition \eqref{e2}
                        encodes the minimal integrability requirements on strong solutions for us to apply Theorem~\ref{thm:adj}, whose verification involves lengthy but straightforward computations. A particular case for which \eqref{e2}
                        holds is that of classical strong solutions.  
                        \item[(ii)] If \(H\) is given by 
                        \eqref{Hpower}, then 
                        \begin{equation*}
                                \begin{aligned}
                                        &\left[\begin{matrix}
                                                A & \frac{1}{2}
                                                \frac{mD^2_{pm}H}{D_{m}H}\\
                                                \frac{1}{2}
                                                \frac{mD^2_{pm}H}{D_{m}H}
                                                & -\frac1{D_mH}
                                        \end{matrix}\right] = \left[\begin{matrix}
                                                A & 0\\
                                                0
                                                & \frac1{g'(m)}
                                        \end{matrix}\right], 
                                \end{aligned}
                        \end{equation*}
                        and
                        \[
                        A=m \left(1+|p|^2\right)^{\frac{\gamma}{2} - 2}  \left[ \mathbb{I}
                        +|p|^2\mathbb{I}
                        -p\otimes p +\left( \gamma -1 \right) p\otimes p \right] + \frac{ \left(1+|p|^2\right)^{{\gamma} - 2}}{g'(m)}p\otimes p. 
                        \]
                        
                        Hence,  recalling that \(\gamma\geq 1\) and arguing as above, if \(g\) is strictly increasing,
                        then  conditions  \eqref{e1} and
                        \eqref{e3} are satisfied for any strong solution to \eqref{eq:mainstationary}
                        satisfying Assumption~\ref{asmp:monotonicity}.
                        
                \end{itemize}
        \end{remark}
        
        \begin{assumption}
                \label{lastone}
                Let $(r,\gamma)$ and $(r_1,\gamma_1)$ be as in Assumption \ref{a-4}, 
                consider a strong solution  $(m,u)$ to Problem \ref{P1}
  in $L^{r_1}\left(\td\right)
                \times W^{1,\gamma_1}\left(\td\right)$, and  let \(A\) and \(\kappa \) be as in Assumption~\ref{A3altbis}. Then,
                $A\in L^{(
                        \frac \gamma 2 )'}(\Tt^d;\Rr^{d\times d})$ and $\kappa\in  L^{(
                        \frac{\gamma^*} 2 )'}(\Tt^d)$.
                
        \end{assumption}
        
        \section{Proof of the main result}    
        
        Now, we prove Theorem \ref{T1}.
        
        \begin{proof}[Proof of Theorem \ref{T1}]
                We split the proof into two sets of conditions: $(a)$ and $(b)$.     
                
                \medskip
                
                \paragraph{\bf Conditions $(a)$}      
                
                Let $(\eta,v)\in C^{\infty}(\Tt^d; \Rr^+_0)\times C^\infty(\Tt^d)$.  We start by observing that   \eqref{eq:weaksol} can
                be written  as
                \begin{align*}
                        \int_{\mathbb{T}^d} (\eta - \tilde{m}) (-v -H(x, Dv, \eta))\,dx
                        +\int_{\mathbb{T}^d} (v-\tilde{u})(\eta - \div(\eta D_pH(x,
                        Dv,  \eta)-1))\,dx \geq 0.
                \end{align*}
                After integration by parts, we obtain
                \begin{align}\label{eq:fromws}
                        &\int_{\mathbb{T}^d} (\eta - \tilde{m}) (-v -H(x, Dv, \eta))\,dx\nonumber
                        \\ &\qquad+\int_{\mathbb{T}^d} \left[(v-\tilde{u})(\eta - 1) +  (Dv-D\tilde{u})\cdot
                        \eta D_pH(x, Dv,  \eta)\right]\,dx \geq 0.
                \end{align}
                On the other hand, by the regularity addressed in Remark \ref{asmp:intergability1}, by Assumption \ref{a-3}, and           
                because $(u, m)$ solves  \eqref{eq:mainstationary} in \(\mathcal D'(\Tt^d)\),
                we have that
                \begin{align}\label{eq:fromss}
                        &\int_{\mathbb{T}^d} (\eta - \tilde{m}) (-u -H(x, Du,    
                        m))\,dx \nonumber
                        \\ &\qquad+\int_{\mathbb{T}^d} \left[(v-\tilde{u})(m - 1) +  (Dv-D\tilde{u})\cdot
                        m D_pH(x, Du,     m)\right]\,dx = 0.
                \end{align}
                Next, we let $\epsi_0>0$ be as in Assumption \ref{a-2}, and 
                 we prove that \eqref{eq:fromss}  and \eqref{eq:fromws} hold for all \((\eta, v)\in L^{r_1}\left(\td;\Rr^+_0\right)
                \times W^{1,\gamma_1}\left(\td\right) \) with
                \(
                \Vert (m,u) - (\eta, v)\Vert_{L^{r_1}\left(\td\right)
                        \times W^{1,\gamma_1}\left(\td\right)} \leq \epsi
                \) for some \(0<\epsi\leq\epsi_0\).   Fix any such pair and use  as test functions in  \eqref{eq:fromws} and \eqref{eq:fromss} a smooth sequence
                $(\eta_k, v_k)$ converging to $(\eta, v)$
                as in Assumption \ref{a-2}.  Then, 
                by Assumption \ref{a-3}, the conclusion follows by taking the limit.
                
                Subtracting the left-hand side of \eqref{eq:fromss} from the left-hand side of \eqref{eq:fromws}, we  obtain
                that                \begin{align*}
                        &\int_{\mathbb{T}^d} (\eta - \tilde{m}) (-(v-u) -H(x, Dv,
                        \eta)+H(x, Du,     m))\,dx \\
                        &\qquad+
                        \int_{\mathbb{T}^d} \big[(v-\tilde{u})(\eta - m) +  (Dv-D\tilde{u})\cdot
                        \left(\eta D_pH(x, Dv,  \eta) - m D_pH(x, Du,     m)\right)\big]\,dx\geq 0,
                \end{align*}
                for all
                \((\eta, v)\in L^{r_1}\left(\td;\Rr^+_0\right)
                \times W^{1,\gamma_1}\left(\td\right) \) with
                \(
                \Vert (m,u) - (\eta, v)\Vert_{L^{r_1}\left(\td\right)
                        \times W^{1,\gamma_1}\left(\td\right)} \leq \epsi
                \) for some \(0<\epsi\leq\epsi_0\).    
                
                Fix  $\bar v\in C^\infty(\td)$ and $\bar \eta\in C^\infty(\td)$, and let \(0<\bar \epsi\leq \epsi_0\) be as in Assumption~\ref{a-1}. For \(\epsi\in\Rr\), set $v_{\eps} := u
                + \varepsilon \bar v$ and  $\eta_{\eps}:= m + \varepsilon \bar \eta$. 
                Using Assumption \ref{asmp:monotonicity}, we have that  
                $\eta_{\eps}\geq 0$ 
                and $\Vert (m,u) -(\eta_\eps, v_\eps)\Vert_{L^{r_1}\left(\td\right)
                        \times W^{1,\gamma_1}\left(\td\right)} \leq \bar \epsi$ for all $\epsi$ small enough. 
                Hence, for any such \(\epsi\), replacing in the above inequality  $\eta$ and $v$ by $\eta_\eps$ and $v_\eps$ yields
                \begin{align*}
                        &\int_{\mathbb{T}^d} (m - \tilde{m} + \varepsilon \bar \eta)\Big
                        (-\varepsilon \bar v -H(x, Dv_{\eps},    \eta_{\eps})+H(x, Du,    m)\Big)\,dx +
                        \int_{\mathbb{T}^d} (u-\tilde{u}+ \varepsilon \bar v)(\varepsilon
                        \bar \eta)\,dx \\
                        &\qquad +
                        \int_{\mathbb{T}^d}   (Du-D\tilde{u} +\varepsilon D\bar v)\cdot  \Big(\eta_{\eps} D_pH(x, Dv_{\eps},
                        \eta_{\eps}) 
                        - m D_pH(x, Du,    m)\Big)dx\geq 0.
                \end{align*}
                
                Because the expression on the left-hand side is zero at $\varepsilon=0$,
                it has a minimum at $\epsi=0$. Hence, if differentiable,  its derivative at $\varepsilon=0$ vanishes.
                The integrand is a differentiable function of $\varepsilon$. 
                Moreover, using Remark~\ref{asmp:intergability1}  and Assumption~\ref{a-1}, we can switch the derivative with the integral sign and pass to the limit, which gives
                \begin{align*}
                        &\int_{\mathbb{T}^d} (m - \tilde{m})\Big (- \bar v  - D_pH(x, Du,
                        m)\cdot D\bar v  - D_mH(x, Du,   m)\bar \eta\Big)\,dx + \int_{\mathbb{T}^d} (u-\tilde{u})\bar \eta\, dx\\                
                        &\quad + \int_{\mathbb{T}^d} (Du-D\tilde{u})\cdot\Big( \bar \eta D_{p}H(x,
                        Du,    m) \\ & \hskip40mm+ m D^2_{pp}H(x, Du,    m)\cdot D\bar v + m D^2_{pm}H(x, Du,    m)\bar \eta
                        \Big)\,dx = 0.
                \end{align*}
                
                By approximation, using Assumption \ref{intA} and the linearity in $\bar \eta$ and $\bar v$,
                we get that the preceding equality holds for all $\bar \eta\in L^r$ and $\bar v\in W^{1,\gamma}$. 
                
                For convenience, we rewrite the previous identity in   terms of $\bar \eta$ and $\bar v$ as follows:
                \begin{equation}
                        \label{weakform}
                        \begin{aligned}
                                &\int_{\mathbb{T}^d} \bigg[-(m - \tilde{m})\bar v  - \Big(
                                (m - \tilde{m})D_pH(x,
                                Du,    m)  - m (Du-D\tilde{u})\cdot D^2_{pp}H(x, Du,    m)\Big) \cdot D\bar
                                v\bigg]
                                dx \\
                                & \quad+ \int_{\mathbb{T}^d} \Big[u-\tilde{u}  - D_mH(x, Du,
                                m)
                                (m - \tilde{m})
                                +  D_{p}H(x, Du,    m)\cdot (Du-D\tilde{u})\\ &\hskip66mm+
                                m D^2_{pm}H(x,
                                Du,    m)\cdot (Du-D\tilde{u})\Big]\bar \eta\, dx = 0.
                        \end{aligned}
                \end{equation}
                
                Accordingly, 
                we can use the above identity with $\bar v=u-\tilde{u}$ and $\bar \eta = m-\tilde{m}$. Then,
                we  integrate by parts to  obtain that
                    \begin{equation}
                	\label{zeroidentity}
                	\begin{aligned}
                        \int_{\mathbb{T}^d}  \Big(- D_mH(x, Du, m) (m - \tilde{m})^2&+
                        (Du-D\tilde{u})\cdot m D^2_{pp}H(x, Du, m)\cdot  (Du-D\tilde{u})\\
                        &+ m D^2_{pm}H(x, Du, m)\cdot  (Du-D\tilde{u})(m - \tilde{m}) \Big)\,dx = 0,
                       \end{aligned}
\end{equation}
                which, together with Assumption~\ref{A3mon},
                implies the uniqueness condition  \eqref{eq:sweq}.            
                
                \medskip   
                
                \paragraph{\bf Conditions $(b)$}         
                
                Assuming \ref{a-4}, \ref{asmp:monotonicity}, \ref{a-2}, \ref{a-3},  and \ref{a-1}, the proof of the previous part of the theorem shows that \eqref{weakform} holds when $\bar v$ and $\bar \eta$ are smooth. Consequently, the following system of equations is satisfied in the weak sense:         
                \begin{equation}\label{eq:uniquenesssystem}
                        \begin{cases}
                                - (m - \tilde{m})+\div\left(D_pH(x, Du,    m)(m - \tilde{m})
                                - m D^2_{pp}H(x, Du,    m) D(u-\tilde{u})\right)\, =\, 0,\\[2mm]       
                                (u-\tilde{u}) - D_mH(x, Du,   m)(m-\tilde{m})
                                +D_{p}H(x, Du,    m)\cdot D(u-\tilde{u})\\\hskip54mm +\, m D^2_{pm}H(x,
                                Du,    m)\cdot D(u-\tilde{u})\,=\, 0.
                        \end{cases}
                \end{equation}
                By Assumption \ref{A3altbis},   $D_mH\neq 0$ almost everywhere.  Accordingly,          
                we can solve the second equation for $m-\tilde{m}$ and use the resulting expression to substitute the $m-\tilde m$ in the first equation. 
                Therefore, we obtain a second-order linear elliptic equation for the difference $u-\tilde{u}$:
                \begin{align}\label{eq:uniqueneselipticeq}
                        &  - \div\left( A(x, Du, m) D(u-\tilde{u}) + b(x, Du, m) (u-\tilde{u})\right)\\\notag
                        &\qquad=\,
                        a(x, Du,    m)(u-\tilde{u})\,+\,c(x, Du,    m)\cdot D(u-\tilde{u}),
                \end{align}
                where \(A\), \(a\), \(b\) and \(c\) are given by \eqref{eq:defcoef}.  
                
                Setting $w = u-\tilde{u}$, we have for   any smooth function $v\in
                C^{\infty}(\td)$ that            
                \begin{equation*}
                        \int_{\td} \left(-\div(A(x, Du, m)^TDv- c(x, Du,    m)v)
                        + b(x, Du, m)\cdot Dv- a(x, Du,    m)v
                        \right)w\,=\,0.
                \end{equation*}
                We then use Theorem~\ref{thm:adj} in the Appendix to find a weak solution, $\bar{v}\in X$, to the adjoint equation,                 \begin{equation}
                        \label{weakformA}
                        -\div(A(x, Du, m)^TD\bar{v} - c(x, Du,    m)\bar{v})
                        + b(x, Du, m)\cdot D\bar{v} - a(x, Du,    m)\bar{v}= \zeta,
                \end{equation}
                such that
                \[
                \int_{\td} A_{ij} \bar{v}_{x_i}\bar{v}_{x_j} + (A^{-1})_{ij}(c_ic_j+{b}_i{b}_j
                + |a|)|\bar{v}|^2 <+\infty,
                \]
                where  $\zeta\in C^{\infty}(\td)$ is arbitrary.                
                By Assumption \ref{lastone}, we have $w\in X$. Hence, we can use $w$ as a test function in \eqref{weakformA}
                to get 
                \[
                \int_{\td}\zeta w=0 \quad \text{for all }  \zeta\in C^{\infty}(\td). 
                \]
                Accordingly $w=0;$ hence, $u=\tilde{u}$, and consequently, $m=\tilde{m}.$
        \end{proof}

\begin{remark}
We note that, while for conditions (b) the existence of a zero-th order term in \eqref{eq:mainstationary} is critical, this not the case for approach (a). In particular, 
the identity	\eqref{zeroidentity} holds for the ergodic mean-field game
	\[ 
		\begin{cases}
			-H\left(x, D u,  m \right)=-\Hh \\
			-\operatorname{div}\left(m D_{p} H\left(x, Du, m\right)\right)=0.
		\end{cases}
\]
Thus, it provides a similar uniqueness result since \eqref{zeroidentity}  implies $m=\tilde m$, $Du=D\tilde u$, from which it follows that the value $\Hh$ is also unique. 
\end{remark}
        
        \appendix        
        
        \section{Appendix: Solvability}\label{sect:A}        
        
        We investigate next the existence of solutions to the  equation
        \begin{equation}
                \label{eq:EDPgen}
                \begin{aligned}
                        Lv=-\div(A(x) D v-c(x) v)+b(x)^T Dv-a(x) v=\zeta
                \end{aligned}
        \end{equation}
        for a given $\zeta\in C^\infty(\Tt^d)$, where \(A\) is a measurable \(\Rr^{d\times d}\)-valued function on \(\Tt^d\),  $b$ and $c$ are  \(\Rr^d\)-valued measurable functions on
        \(\Tt^d\), and \(a\) is a real-valued measurable function on \(\Tt^d\).  Our arguments closely follow those in \cite[Theorem~3.2]{Trudinger1973}  regarding a boundary-value problem on an open set. 
        
        We start by defining a space $X$ as the space
        of locally integrable functions, $u$,  whose derivative in the sense of distributions, $Du$, is a function that satisfies
        the following integrability condition
        \begin{equation}\label{eq:normX}
                \|u\|_X^2=\int_{\td} \left(Du^TA Du +
                \kappa |u|^2\right) dx <+\infty, 
        \end{equation}
        where 
        \begin{equation}
                \label{eq:kappa}
                \begin{aligned}
                        \kappa=  c^TA^{-1}c+b^TA^{-1}b + |a|.
                \end{aligned}
        \end{equation}
        
        Next, we introduce the assumptions on the coefficients under which we prove the aforementioned existence result.
        \begin{assumption}\label{ass:onA}
                There exist a  measurable function, \(\sigma:\Tt^d\to[0,\infty)\), and  \(\tfrac12 \leq q <1\)   such that
                \begin{equation*}
                        \begin{aligned}
                                \sigma^{\frac{q}{q-1}} \in L^{1}(\Tt^d)
                        \end{aligned}
                \end{equation*}
                and, for all \(\xi\in \Rr^d\) and a.e.~on \(\Tt^d\), we have
                \begin{equation*}
                        \begin{aligned}
                                \xi^T A\xi \geq \sigma|\xi|^2.
                        \end{aligned}
                \end{equation*}
        \end{assumption}   
        
        \begin{remark}\label{rmk:posdef}
                Note that Assumption~\ref{ass:onA} yields \(\sigma>0\) a.e.~in \(\Tt^d\) because \(\frac{q}{q-1}<0\).
                Moreover, the inverse map of \(A\), \(A^{-1} :\Tt^d\to\mathbb{R}^{d\times d}\),
                is well defined. If, in addition, $A$ is  a.e.~symmetric and positive definite on \(\Tt^d\), then both \(A\) and \(A^{-1}\) admit a unique symmetric and positive definite square-root matrices, which we denote by \(A^{\frac12}\) and \(A^{-\frac12}\), respectively.
                Observe further that in expressions like \eqref{eq:normX} and \eqref{eq:kappa}, we can replace $A$ by its symmetric part $A_S=\frac{A + A^T}{2}$. However, for some estimates in the nonsymmetric case, we  need to control  \( ((A^{-1})_S)^{-1} \) in terms of    \(  
                A_S\) as detailed in the following assumption.  
        \end{remark}
        
        \begin{assumption}\label{ass:onA1} There exists \(\tau>0\) such that , for all \(\xi\in \Rr^d\) and a.e.~on \(\Tt^d\), we have
                \[ \xi^T((A^{-1})_S)^{-1} \xi\leq \tau
                \xi^TA_S\xi.\] 
        \end{assumption}   
        \begin{remark}
                In the symmetric case, the preceding assumption holds with equality and $\tau=1$. 
        \end{remark}
        
        For $1\leq p<d$, the Sobolev conjugated exponent of \(p\) is $p^*$ given by $\frac 1 {p^*}=\frac 1 p - \frac 1 d$. 
        
        \begin{assumption}\label{ass:onkappa} The function \(\kappa\) in \eqref{eq:kappa} is a.e.~positive on \(\Tt^d\) and there exists \(\beta\geq 1\) such that \(\kappa \in L^\beta(\Tt^d)\) and
                \begin{equation*}
                        \begin{aligned}
                                2{\beta'} \leq (2q)^*,
                        \end{aligned}
                \end{equation*}
                where \(q\) is as in Assumption~\ref{ass:onA}.
        \end{assumption} 
        
        \begin{remark} If \(\frac{d}{d+2} \leq q <1\), then \(\frac{(2q)^*}{2} \geq 1\) and Assumption~\ref{ass:onkappa} is satisfied for any large enough \(\beta\).
        \end{remark} 
        
        \begin{assumption}\label{ass:onkernel} If \(u\in X\) is such that \(Lu=0\), meaning
                \begin{equation}
                        \label{a7e}     
                        \begin{aligned}
                                \int_{\td} \left(Du^TA Du +u
                                \left(c^T+b^T\right)Du -a|u|^2\right) dx = 0,
                        \end{aligned}
                \end{equation}
                then \(u=0\).
        \end{assumption} 
        
        \begin{remark}\label{rmk:zeroLu} Assume that Assumption~\ref{ass:onA} holds. Then, if there exist \(\epsilon>\frac12\) and \(\theta>0\) such that
                \begin{equation}\label{eq:forcoerc}
                        \begin{aligned}
                                \epsilon \left( c^T A^{-1} c + b^T A^{-1} b\right) + a \leq -\theta
                        \end{aligned}
                \end{equation}
                a.e.~on \(\td\), then Assumption~\ref{ass:onkernel} also holds.
                In fact, combining Assumption~\ref{ass:onA}  with Cauchy's inequality and Remark 
                \ref{rmk:posdef}, we have that
                \begin{equation*}
                        \begin{aligned}
                                0&=\int_{\td} \left(Du^TA Du +u
                                \left(c^T+b^T\right)Du -a|u|^2\right) dx \\
                                & =\int_{\Tt^d} \left(Du^TA_S Du -uA^{-\frac{1}2}_S c\cdot A^{\frac{1}2}_SDu
                                +uA^{-\frac{1}2}_S b\cdot A^{\frac{1}2}_S Du -
                                a|u|^2\right) dx\\
                                &\geq\left(1-\frac1{2\epsilon}\right) \int_{\Tt^d} Du^TA Du\,dx -  \int_{\Tt^d} \left[\epsilon \left( c^T A^{-1} c + b^T A^{-1} b\right) + a\right]|u|^2 dx \\
                                &\geq \left(1-\frac1{2\epsilon}\right) \int_\td \sigma |Du|^2 dx +\ \theta\int_\td |u|^2 dx,
                        \end{aligned}
                \end{equation*}
                from which we get \(u=0\). We further observe that we must have \(a<0\) a.e.~on \(\td\)  for \eqref{eq:forcoerc}
                to hold because \(A^{-1}\) is positive definite. 
        \end{remark}
        
        As we prove next, Assumptions~~\ref{ass:onA} and \ref{ass:onkappa} yield some embeddings that are useful in the sequel.   
        
        \begin{lem}\label{lem:LsigmaL}
                If Assumption~\ref{ass:onA} holds, then \(L^2(\Tt^d,\sigma(\cdot))
                \)  is continuously embedded in  \(L^{2q}(\Tt^d)\). Precisely,
                \begin{equation*}
                        \begin{aligned}
                                \int_{\Tt^d} |v|^{2q} dx \leq \left(\int_{\Tt^d} \sigma^{\frac{q}{q-1}}\right)^{1-q}
                                \left(\int_{\Tt^d} \sigma |v|^2 dx\right)^q.
                        \end{aligned}
                \end{equation*}
        \end{lem}
        
        \begin{proof} Let  \(v\in L^2(\Tt^d,\sigma(\cdot))
                \). Then, using H\"older's inequality with \(r=q^{-1}\) and Assumption~\ref{ass:onA}, we have that
                \begin{equation*}
                        \begin{aligned}
                                \int_{\Tt^d} |v|^{2q} dx &=\int_{\Tt^d} \sigma^{-q} \sigma^q |v|^{2q} dx \leq \left(\int_{\Tt^d} \sigma^{-qr'} dx \right)^\frac{1}{r'}
                                \left(\int_{\Tt^d} \sigma^{qr} |v|^{2qr} dx\right)^{\frac1r} \\&= \left(\int_{\Tt^d} \sigma^{\frac{q}{q-1}} dx \right)^{1-q}
                                \left(\int_{\Tt^d} \sigma |v|^2 dx\right)^q,
                        \end{aligned}
                \end{equation*}
                which concludes the proof.
        \end{proof}
        
        \begin{lem}\label{lem:LinLk}
                If Assumption \ref{ass:onkappa} holds, then
                \(L^{(2q)^*}(\Tt^d)\)  is continuously embedded in   \(L^2(\Tt^d,\kappa(\cdot))\).
                Precisely,
                \begin{equation*}
                        \begin{aligned}
                                \int_{\Tt^d} \kappa|v|^{2} dx \leq  \left(\int_{\Tt^d} \kappa^{\beta} dx \right)^{\frac1\beta}
                                \left(\int_{\Tt^d}  |v|^{(2q)^*} dx\right)^{\frac2{(2q)^*}}.
                        \end{aligned}
                \end{equation*}
                
        \end{lem}
        
        \begin{proof}
                Let    \(v\in L^{(2q)^*}(\Tt^d)\). By H\"older's inequality,
                \begin{equation*}
                        \begin{aligned}
                                \int_{\Tt^d} \kappa|v|^{2} dx  \leq \left(\int_{\Tt^d} \kappa^{\beta } dx\right)^{\frac1\beta}
                                \left(\int_{\Tt^d}  |v|^{2\beta'} dx\right)^{\frac1{\beta'}}\leq
                                \left(\int_{\Tt^d} \kappa^{\beta } dx\right)^{\frac1\beta}
                                \left(\int_{\Tt^d}  |v|^{(2 q)^*} dx\right)^{\frac2{(2q)^*}}, 
                        \end{aligned}
                \end{equation*}
                using Assumption \ref{ass:onkappa}.
        \end{proof}
        
        \begin{corollary}\label{cor:compact}
                If Assumptions~\ref{ass:onA} and \ref{ass:onkappa} hold, then
                \(X\)  is compactly embedded in   \(L^2(\Tt^d,\kappa(\cdot))\).
        \end{corollary}
        
        \begin{proof}
                By Lemma~\ref{lem:LinLk} and the compact embedding of \(W^{1,2q}(\Tt^d)\) into \(L^{(2q)^*}(\Tt^d)\), it remains to  show that a bounded sequence in \(X\) is also bounded in \(W^{1,2q}(\Tt^d)\) because the composition between a continuous and  a compact
                embedding is a compact embedding.   
                
                Let \((u_n)_n\) be a bounded sequence in \(X\). Then, recalling \eqref{eq:normX} and using Assumption~~\ref{ass:onA}, we have that 
                \begin{equation}\label{eq:unbdd}
                        \begin{aligned}
                                \sup_n \int_{\Tt^d} \sigma |Du_n|^2 dx<\infty \quad \text{and } \quad \sup_n \int_{\Tt^d} \kappa |u_n|^2 dx<\infty.
                        \end{aligned}
                \end{equation}
                Then, setting \(w_n = u_n - \int_{\Tt^d} u_n dx\), the Poincar\'e--Wirtinger inequality combined with the preceding estimate and Lemma~\ref{lem:LsigmaL} applied to \(v=D w_n=Du_n\) yields 
                \begin{equation*}
                        \begin{aligned}
                                \sup_n \Vert w_n\Vert_{W^{1,2q}(\Tt^d)} <\infty.
                        \end{aligned}
                \end{equation*}
                Moreover, using  Lemma~\ref{lem:LinLk} and the compact embedding of \(W^{1,2q}(\Tt^d)\)
                into \(L^{(2q)^*}(\Tt^d)\),
                \begin{equation}\label{eq:meanwn}
                        \begin{aligned}
                                \sup_n
                                \int_{\Tt^d} \kappa |w_n|^2 dx<\infty.
                        \end{aligned}
                \end{equation}
                Consequently, using the integrability and positivity of \(\kappa\),
                \begin{equation*}
                        \begin{aligned}
                                \sup_n \left|\int_{\Tt^d} u_n dx\right| &\leq 1 + \sup_n \left|\int_{\Tt^d} u_n dx\right|^2 = 1 + \left(\int_{\Tt^d} \kappa dx\right)^{-1} \sup_n \int_{\Tt^d} \kappa \left|\int_{\Tt^d}
                                u_n dx\right|^2 dx\\
                                &= 1 + \left(\int_{\Tt^d} \kappa dx\right)^{-1} \sup_n \int_{\Tt^d}
                                \kappa \left|u_n - w_n\right|^2 dx <\infty
                        \end{aligned} 
                \end{equation*}
                by \eqref{eq:unbdd} and \eqref{eq:meanwn}. Finally, arguing as above, the preceding estimate, \eqref{eq:unbdd},  Lemma~\ref{lem:LsigmaL},
                and the Poincar\'e--Wirtinger
                inequality imply that \((u_n)_n\) is bounded in  \(W^{1,2q}(\Tt^d)\).     
        \end{proof}
        
        So far, we have seen that if Assumptions~~\ref{ass:onA} and \ref{ass:onkappa} hold, then    \(L^2(\Tt^d,\kappa(\cdot))\) is a Hilbert space that compactly contains \(X\).
        This embedding motivates us to introduce a bilinear form associated with $L$ on $X$; precisely, let  \(B: X\times X \to \Rr\) be the bilinear form given by
        \begin{equation}\label{eq:bilinearB}
                \begin{aligned}
                        B[u,v]=\int_{\Tt^d} \left(Dv^TADu -vc^TDu +ub^TDv -
                        a u v\right) dx\quad \text{ for \((u,v) \in X\times X,\)} 
                \end{aligned}
        \end{equation}
        which, as proved below, is bounded in \(X\).
        
        \begin{lem}\label{lem:Xbdd}
                Assume that Assumptions~~\ref{ass:onA}, \ref{ass:onA1}, and \ref{ass:onkappa}
                hold, and consider the bilinear form \(B\) introduced in \eqref{eq:bilinearB}. Then, we have for all  \(u,v\in X\) that
                \begin{equation*}
                        \begin{aligned}
                                |B[u,v]|\leq (\tau+3)\|u\|_X\|v\|_X.
                        \end{aligned}
                \end{equation*}
        \end{lem}
        
        \begin{proof}
                Let \(u,v\in X\). Then, using Assumptions~~\ref{ass:onA}, \ref{ass:onA1}, and \ref{ass:onkappa}
                and
                recalling Remark~\ref{rmk:posdef}, it follows that
                \begin{align*}
                        \big|B[u,v]\big|&\leq \int_{\Tt^d} \left(|Dv\cdot AD u| +|vc\cdot Du| +|u
                        b\cdot Dv| +
                        |a u v|\right) dx \\
                        &=  \int_{\Tt^d} \Big(|(A_S)^{\frac12}Dv\cdot (A_S)^{-\frac12}A Du| +|v(A_S)^{-\frac12}c\cdot
                        (A_S)^{\frac12}Du| \\&
                        \hskip25mm+|u(A_S)^{-\frac12} b\cdot (A_S)^{\frac12}Dv| +
                        |a|^{\frac12}| |u| |a|^{\frac12}||v|\Big) dx\\
                        &\leq \bigg( \int_{\Tt^d} |A^{\frac12}Dv|^2dx\bigg)^{\frac12}\bigg( \int_{\Tt^d}
                        |(A_S)^{-\frac12}A Du|^2dx\bigg)^{\frac12} \\&\quad+  \bigg( \int_{\Tt^d} |v(A_S)^{-\frac12}c|^2dx\bigg)^{\frac12}\bigg(
                        \int_{\Tt^d} |(A_S)^{\frac12}Du|^2dx\bigg)^{\frac12}
                        \\&\quad+ \bigg( \int_{\Tt^d} |u(A_S)^{-\frac12} b|^2dx\bigg)^{\frac12}\bigg(
                        \int_{\Tt^d} |(A_S)^{\frac12}Dv|^2dx\bigg)^{\frac12}
                        \\&\quad+ \bigg( \int_{\Tt^d} |a||u|^2dx\bigg)^{\frac12}\bigg( \int_{\Tt^d} |a||v|^2dx\bigg)^{\frac12}\\
                        &\leq (\tau+3)\|u\|_X\|v\|_X.\qedhere
                \end{align*}
        \end{proof}
        
        Next, we show that there is a coercive shift of \(B\) in \(X\).
        
        \begin{lem}\label{lem:Xcoerc} 
                Assume that Assumptions~~\ref{ass:onA} and \ref{ass:onkappa}
                hold, and consider the bilinear form \(B\) introduced in \eqref{eq:bilinearB}.
                Then, we have for all   \(u\in X\)  that
                \[
                B[u,u]\geq \frac12 \|u\|_X^2 -\frac 3 2 \int_{\Tt^d} \kappa| u|^2 dx.
                \]
        \end{lem}
        \begin{proof}
                Let \(u\in X\). Then, using Cauchy's inequality
                and recalling Remark~\ref{rmk:posdef}, we have that
                \begin{align*}
                        B[u,u]&=\int_{\Tt^d} \left(Du^TA Du -uA^{-\frac{1}2} c\cdot A^{\frac{1}2}Du +uA^{-\frac{1}2} b\cdot A^{\frac{1}2}Du -
                        a|u|^2\right) dx\\
                        &\geq\frac12 \int_{\Tt^d} Du^TA Du\,dx -  \int_{\Tt^d} \big(|uA^{-\frac{1}2} c|^2 + \,|uA^{-\frac{1}2}
                        b|^2+ a|u|^2\big)dx\\
                        &=\frac12 \|u\|_X^2 -\frac 3 2 \int_{\Tt^d} \kappa| u|^2 dx. \qedhere
                \end{align*}
                
        \end{proof}
        
        \begin{lem}\label{lem:bij} 
                Assume that Assumptions~~\ref{ass:onA}, \ref{ass:onA1}, and \ref{ass:onkappa}
                hold, and   consider the operator $L_\upsilon:X\to X'$  defined for $\upsilon>0$ by $L_\upsilon=L+\upsilon\kappa$. If \(\upsilon>\frac 3 2\), then \(L_\upsilon\) is a bijection between $X$ and $X'$. 
        \end{lem}
        \begin{proof}
                This lemma is a direct consequence of the Lax--Milgram theorem in view of Lemmas~\ref{lem:Xbdd} and  \ref{lem:Xcoerc}. 
        \end{proof}
        
        Now, we can address the existence of solutions to \eqref{eq:EDPgen}.
        \begin{theorem}\label{thm:adj}
                Assume that Assumptions~~\ref{ass:onA}, \ref{ass:onA1},  \ref{ass:onkappa}, and \ref{ass:onkernel}
                hold, and let \(\zeta \in X'\). Then, there exists a unique \(u\in X\) solving  \eqref{eq:EDPgen}; that is, for all \(v\in X\), we have that
                \begin{equation*}
                        \begin{aligned}
                                \int_{\Tt^d} \left(Dv^TADu -vc^TDu +ub^TDv -
                                a u v\right) dx = \langle \zeta,v\rangle_{X',X}.
                        \end{aligned}
                \end{equation*}
        \end{theorem}
        
        \begin{proof}
                We start by observing that Corollary~\ref{cor:compact} gives that the compactness of the operator \(J:X \to X'\) defined for \(u\in X\) by \(J[u]= \kappa u$, with \(\langle\kappa u, v\rangle_{X',X} = \int_{\Tt^d} \kappa uv dx\).  In fact, if \((u_n)_{n\in\Nn}\) is a bounded sequence in \(X\), then, up to extracting a subsequence, \(u_n \to u\) in \(L^2(\Tt^d,\kappa(\cdot))\) for some \(u\in L^2(\Tt^d,\kappa(\cdot))\)  by  Corollary~\ref{cor:compact}. Recalling the definition of \(\Vert \cdot\Vert_X\) in \eqref{eq:normX}, we easily deduce that \(u\in X\) by  Assumption~~\ref{ass:onA}. Thus, H\"older's inequality yields 
                for every \(v\in X\) that
                \begin{equation*}
                        \begin{aligned}
                                \Vert  \kappa u_n -  \kappa u \Vert_{X'} &= \sup_{v\in X, \Vert v\Vert_X\leq 1} \bigg|\int_{\Tt^d} \big(\kappa^{\frac12}(x) u_n(x)-\kappa^\frac12(x)u(x)\big) \kappa^\frac12v(x) dx\big|\\ &\leq\Vert u_n - u\Vert_{L^2(\Tt^d,\kappa(\cdot))} \Vert v\Vert_{L^2(\Tt^d,\kappa(\cdot))} \to 0.
                        \end{aligned}
                \end{equation*}
                
                Consequently, using Lemma~\ref{lem:bij},
                the operator \(K = L_{\upsilon_0}^{-1} \circ J:X\to X\) is a linear and compact operator  for any \(\upsilon_0>2\) fixed. Then,  applying the Fredholm alternative for compact operators, we conclude that  either
                \begin{align}
                        &\text{the equation \(u - \upsilon_0Ku=L_{\upsilon_0}^{-1}\zeta
                                \) has a unique solution in \(X\)} \label{eq:FH1}\\
                        \text{or}\notag\\
                        &\text{the homogeneous equation  \(u - \upsilon_0Ku=0\) has a solution
                                \(u\not=0\) in \(X\).}\label{eq:FH2}
                \end{align}
                
                We claim that \eqref{eq:FH2} does not hold. In fact, because \(L_{\upsilon_0}\) is bijective by Lemma~\ref{lem:bij}, \(u -\upsilon_0Ku=0\) is equivalent to \(Lu=0\), and the claim follows by Assumption~\ref{ass:onkernel}.  Thus,  \eqref{eq:FH1} must hold, which together with the bijectivity of  \(L_{\upsilon_0}\)
                provides a   solution to    \eqref{eq:EDPgen}. Moreover, such a solution is unique.    \end{proof}

        \bibliographystyle{plain}
\def\cprime{$'$}

\end{document}